\documentclass[12pt,english]{extarticle}
\usepackage{tgpagella}
\usepackage{tgheros}
\usepackage{tgcursor}
\usepackage{appendix}
\usepackage[T1]{fontenc}
\usepackage[latin9]{inputenc}
\usepackage{geometry}
\usepackage{color}
\usepackage{babel}
\usepackage{latexsym}
\usepackage{pifont}
\usepackage{units}
\usepackage{amsmath}
\usepackage{amsthm}
\usepackage{amssymb}
\usepackage{stackrel}
\usepackage{graphicx}
\usepackage[authoryear]{natbib}
\usepackage[colorlinks=true, linkcolor=blue, citecolor=blue,  filecolor=blue, urlcolor=blue]{hyperref}

\makeatletter

\theoremstyle{plain}
\newtheorem{thm}{Theorem}
  \theoremstyle{plain}
  \newtheorem{lem}[thm]{Lemma}
  \theoremstyle{plain}
  \newtheorem{cor}[thm]{Corollary}
  \theoremstyle{remark}
  \newtheorem{rem}[thm]{Remark}
  \theoremstyle{remark}
  
  \theoremstyle{plain}
  \newtheorem{prop}[thm]{Proposition}
  \theoremstyle{definition}
  
  \theoremstyle{plain}

\usepackage{graphicx}
\usepackage{bbm}
\DeclareMathAlphabet{\mathsl}{OT1}{cmss}{m}{sl}
\SetMathAlphabet{\mathsl}{bold}{OT1}{cmss}{bx}{sl}

\newcommand{\one}[1]{\mathbbm{1}_{\left\{#1\right\}}}

\renewcommand{\theta}{\vartheta}

\DeclareMathOperator{\R}{\mathbb{R}}

\DeclareMathOperator{\Z}{\mathbb{Z}}
\DeclareMathOperator{\N}{\mathbb{N}}

\DeclareMathOperator{\De}{d}

\newcommand{\e}{\mathrm{e}}
\newcommand{\cO}{\mathcal{O}}

\newcommand{\cH}{\mathcal{H}}

\numberwithin{equation}{section}
\newcommand{\eq}[1]{\begin{equation#1}}
\newcommand{\eeq}[1]{\end{equation#1}}

\global\long\def\prob{\mathsf{P}}

\global\long\def\ER{\mathbb{E}}  %%% Law of random walk
\bibpunct{(}{)}{,}{}{,}{,}
\date{}

\AtBeginDocument{

}

\makeatother

\begin{document}

\title{A note on the Green's function for the transient random walk without killing on the half lattice, orthant and strip}
\author{Alberto Chiarini \thanks{Universit\'e d'Aix-Marseille, Marseille, France\hfill  \texttt{alberto.chiarini@univ-amu.fr}}\qquad Alessandra Cipriani    \thanks{Weierstrass Institute, Berlin, Germany\hfill  \texttt{Alessandra.Cipriani@wias-berlin.de}}}

\maketitle
\begin{abstract}In this note we derive an exact formula for the Green's function of the random walk on different subspaces of the discrete lattice (orthants, including the half space, and  the strip) without killing on the boundary in terms of the Green's function of the simple random walk on $\Z^d$, $d\ge 3$.
\end{abstract}
\section{Introduction}
The literature encompassing random walks on subgraphs of the square lattice is very rich, spanning not only probability theory, but also combinatorics, queueing theory, and algebraic geometry (\cite{bostan,BMS,DenisovWachtel,FIMM, KurMal, Raschel,Uchi10} to mention only a few). In this short note we focus on one particular observable of the random walk, the Green's function, which measures the local time of the walk \cite[Chapter 4]{LawlerLimic}. In this short note we answer the natural question of whether this quantity is directly related to the Green's function $g(\cdot,\,\cdot)$ of the simple random walk on the whole lattice $\Z^d$. We will be concerned with the transient case, that is, when $g$ is finite, although our formulas can be derived in the recurrent setting adding an extra killing to the walk. To the best of the authors' knowledge, explicit formulas for the Green's function were obtained only in the case when a killing is imposed on the boundary of the graph, for example on the axes (\citet[Chapter 8]{LawlerLimic}) or for walks with Neumann and reflected boundary conditions (\cite{GanPer} study for example the scaling limit of reflected random walks in a planar domain). We obtain a closed formula
for the Green's function in any subspace which is the intersection of $m$ hyperplanes, $m\leq d$ in $d\ge 3$, and for the strip of fixed width in $d\ge 4$. Using a simple ``folding'' technique, we fold $\Z^d$ onto each of these subgraphs, and by electric networks reduction we deduce a representation formula exclusively in terms of $g$, which enables also to approximate numerically the Green's function in each of these subgraphs by means of Bessel functions.
\paragraph{Structure of the paper}After introducing some notation in Section~\ref{sec:notation}, we give the explicit formulas for the Green's function of the half space in Section~\ref{sec:half_space_prel}, the strip in Section~\ref{sec:strip}, and of the orthant in Section~\ref{sec:orthant_prel}.

\section{General setup}\label{sec:notation}

Let $\mathcal{G}=(V,E)$ be a connected graph of bounded degree with vertex set $V$ and edge set $E$. We will write $x\sim y$ if $\{x,y\}\in E$. We endow each edge $\{x,y\}\in E$ with a positive and finite conductance $c_\mathcal{G}(x,y)$ and for each $x\in V$ we write $\pi_\mathcal{G}(x):= \sum_{y\sim x} c_\mathcal{G}(x,y)$.

Let $(S_n)_{n\in \N_{0}}$ be a discrete time random walk on $\mathcal{G}$ with transition probability
\[
\prob(S_{n+1}=y|S_n=x) = \frac{c_\mathcal{G}(x,y)}{\pi_\mathcal{G}(x)}.
\]
Then $S_n$ is a reversible, irreducible Markov chain on $\mathcal{G}$ with stationary measure given by $\pi_\mathcal{G}$. If the random walk is transient, then we can define the Green's function
\begin{equation}\label{eq:green}
  G_\mathcal{G}(x,y) = \frac{1}{\pi_{\mathcal{G}}(y)} \ER_x \left[\sum_{m\geq0}\one{S_{m}=y}\right],\quad x,\,y\in V,
\end{equation}
where $\ER_x$ is the expectation with respect to the random walk $(S_n)_{n\in \N_{0}}$ started at $x\in V$.
It is easy to see that  $G_\mathcal{G}(x,y) = G_\mathcal{G}(y,x)$, being the walk $S_n$ reversible with respect to $\pi_{\mathcal{G}}$.

We will adopt a special notation when the graph has vertex set $\Z^d$, edge set $\{\{x,y\}: \|x-y\|=1\}$ and unitary conductances. In this case we are just looking at the classical simple random on $\Z^d$ which is transient for $d\geq 3$.  We will denote its Green's function simply by $g(x,y)$, $x,y\in \Z^d$. Notice that using \eqref{eq:green} $g(x,y)$ differs for a normalization constant of value $2d$ from the more classical definition $\widetilde{g}(x,y):= \ER_x[\sum_{m\geq0}\one{S_{m}=y}]$.

\paragraph*{Notation.}
Let  $\left(\mathbf{e}_{(i)}\right)_{i=1,\,\ldots,\,d}$
denote the canonical basis of $\R^{d}$. For a vector $v\in\R^{d}$
we use also the notation $v=\left(v_i\right)_{i=1}^{d}$ to specify its components and for a vector-valued process $X$ we specify its components with $(X_n)_{n\ge0}=(X_n^{(1)},\,\ldots,\,X_n^{(d)})_{n\ge 0}$ .
We denote $\N = \{1,2,\dots\}$ and $\N_{0} = \{0\}\cup \N$.

\section{Green's function on the half lattice\label{sec:half_space_prel}}

The half lattice $\cH$  is the graph with vertex set $H:=\{x\in\Z^{d}:x_{1}\geq0\}$
and edge set $E:=\{\{x,y\} :\|x-y\|=1,\,x,\,y\in H\}$. We set all the conductances equal to one, so that $\pi_\cH (x) = \mathrm{deg}(x)$. The Green's function of the simple random walk $(S_{m})_{m\ge0}$
on $\cH = (H,E)$ is simply given, by means of \eqref{eq:green}, by
\[
G_{\cH}(x,\,y):=\frac{1}{\pi_{\cH}(y)}\ER_{x}\left[\sum_{m\geq0}\one{S_{m}=y}\right],\quad x,y\in H.
\]
In the case in which one considers a random walk on $\cH$ with killing on $\{0\}\times \Z^{d-1}$, the Green's function has the form
$$g(x,y)-g(x,\overline{y})\quad x,\,y\in H,$$
where $\overline{\cdot}$ is the map which takes $y=\left(y_1,\,y_2,\,\ldots,\,y_d\right)\in \Z^d$ to $\overline y:=(-y_1,\,y_2,\,\ldots,\,y_d)$ (see \citet[Proposition 8.1.1]{LawlerLimic}). We compare this formula with our result that considers the case without killing.
\begin{prop}[Green's function on the half-space]
\label{lem:half_green}We have, for all $x,\,y\in H,$
that
\begin{equation}\label{eq:green_halfspace}
G_{\cH}(x,\,y)=g(x,\,y)+g\left(x,\,\overline{y}-\mathbf e_{(1)}\right).
\end{equation}
\end{prop}
\begin{proof} We will work in several steps by reducing our problem
from considering a random walk on the half space to one on $\Z^{d}$. The idea is basically to fold $\Z^d$ on itself along the line $\{x:\,x_1=-\nicefrac{1}{2}\}$ to obtain a graph which looks like the half lattice plus some additional lateral ``combteeth'', and thus obtain a half space with reflection across the vertical axis. We will explain this now in mathematical terms.
\begin{figure}
\begin{center}
\includegraphics[width=0.4\textwidth]{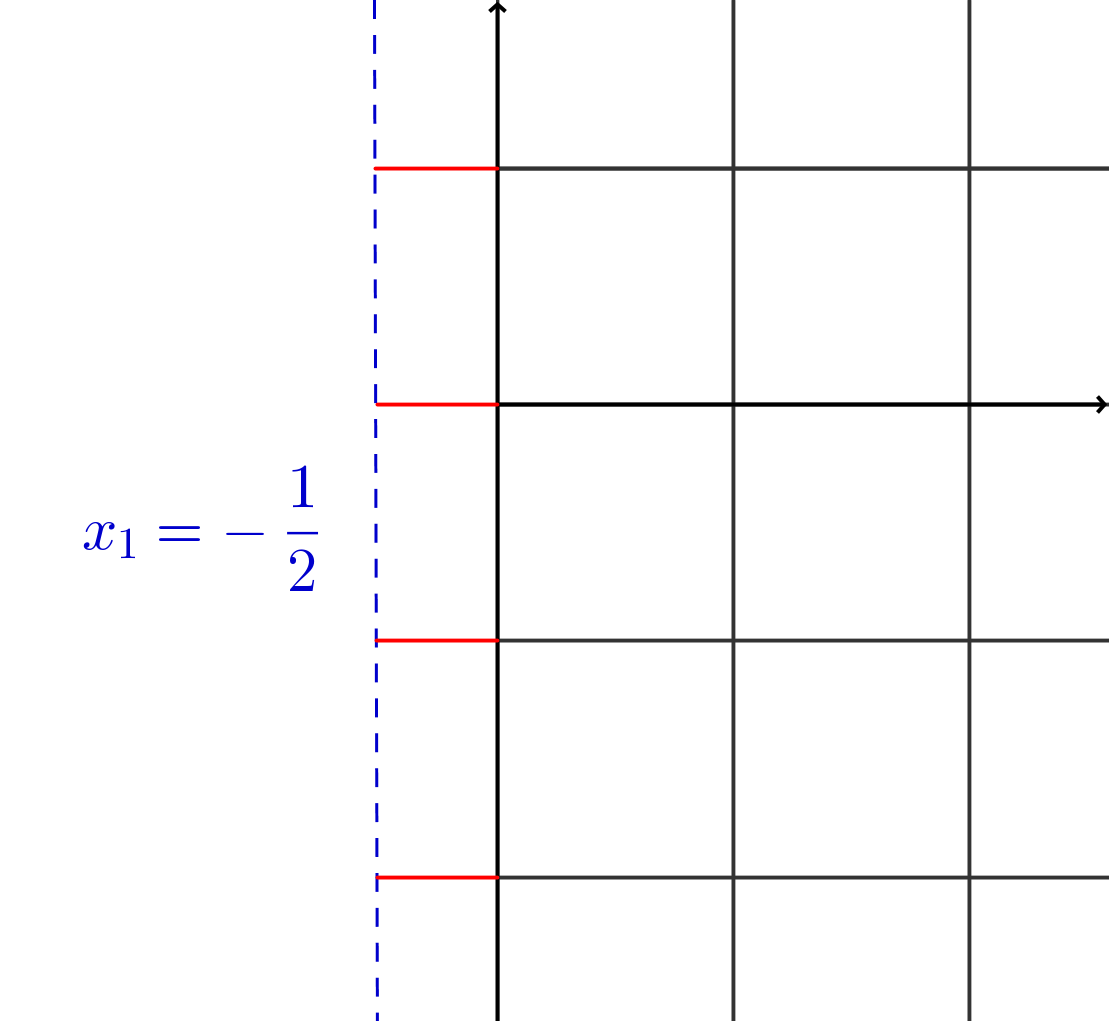}
\end{center}
\caption{A portion of $\mathcal{H}'$ with in red the conductances with value $2$. }
\label{fig:H'}
\end{figure}

Let us begin by adding to $\cH$ all the bonds $\{z,z-\nicefrac{1}{2}\,\mathbf{e}_{(1)}\}$ for all $z\in\{0\}\times\Z^{d-1}$. Call this new graph $\cH'$.
Let us put for each edge a conductance
\[
c_{\cH'}(x,\,y):=\begin{cases}
2 & \left\Vert x-y\right\Vert =\nicefrac{1}{2}\\
1 & \mathrm{otherwise}
\end{cases}.
\]
(see Figure \ref{fig:H'} for a two-dimensional example). It is easy to see that $G_{\cH'}(x,y) = G_\cH(x,y)$ for all $x,y\in H$ as there is no current flowing through the new bonds and the old ones are unchanged. Also denote by $(L_n)_{n\geq 0}$ the random walk on $\cH'$ with transition probabilities given by $p_{xy}:= c_{\cH'}(x,y)/\pi_{\cH'}(x)$.

Consider further the graph obtained from $\Z^{d}$ by splitting the conductance on the bond $\{z-\mathbf{e}_{(1)},z\}$ with $z\in \{0\}\times \Z^{d-1}$ into two conductances in series on the bonds $\{z- \mathbf{e}_{(1)}, z-\nicefrac{1}{2}\,\mathbf{e}_{(1)}\}$ and $\{z-\nicefrac{1}{2}\, \mathbf{e}_{(1)},z\}$. More precisely on this new graph, which we call $\mathcal{Q}$,  put the following conductances:
\[
c_\mathcal{Q}(x,\,y):=\begin{cases}
2 & \left\Vert x-y\right\Vert =\nicefrac{1}{2}\\
1 & \mathrm{otherwise}
\end{cases},\quad x,\,y\in \Z^d\cup \left( \{-\nicefrac{1}{2}\}\times \Z^{d-1}\right).
\]
By Ohm's law of conductances in series, this ensures that the new graph obtained is equivalent to $\Z^{d}$. More precisely $g(x,y) = G_\mathcal{Q}(x,y)$ for all $x, y\in \Z^d$.
\begin{figure}
\begin{center}\includegraphics[width=0.4\textwidth]{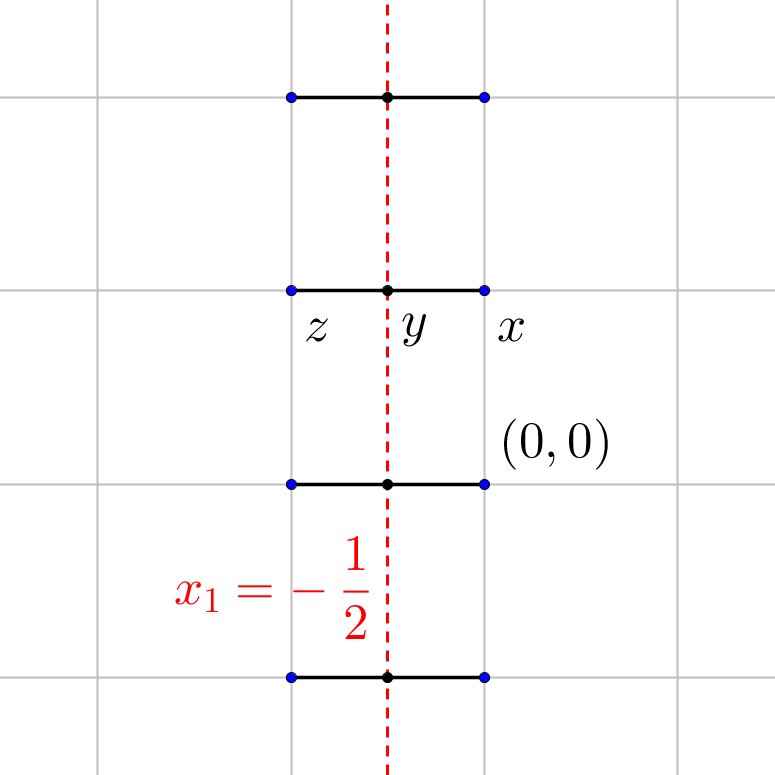}
\end{center}
\caption{A portion of $\mathcal{Q}.$ Starting from $\Z^{d}$ (light gray
lines), we split the conductance $\{x,z\}$ which has value one in the two conductances $\{x,\,y\}$ and $\{y,\,z\}$ with value two.}
\label{fig:Q}
\end{figure}

 Consider the simple random walk $W=(W_{n})_{n\geq0}$ on $\mathcal{Q}$ with transition probabilities given by  $q_{xy} :=c_{\mathcal{Q}}(x,\,y)/\pi_\mathcal{Q}(x)$ and started in $H$. Write  $W_n=\left(U_{n},\,V_{n}\right)$ with $U_{n}$ being the projection
of $W_n$ on the first coordinate direction and $V_{n}$
the projection of $W_n$ on the remaining $d-1$ components.
Finally, consider $W'=(W'_{n})_{n\ge0}$ which is the reflection of
$W$ with respect to the hyperplane $\left\{ x\in \R^d :\,x_{1}=-\nicefrac{1}{2}\right\}$.
In other words, $W'_{0}=W_0$ and $W'_{n}=\left(-U_{n}-1,\,V_{n}\right)\one{U_{n}\le-\nicefrac{1}{2}}+(U_{n},\,V_{n})\one{U_{n}>-\nicefrac{1}{2}}$.

%. If both $x_{1}\ge0$ and $y_{1}\ge0$ then the step
%probabilities are, in both cases, those of a simple random walk on
%$\Z^{d}$. If, without loss of generality, $x_{1}=-\nicefrac{1}{2}$,
%\begin{align*}
%q_{xy} & =\proR_{x}\left(W'_{1}=y\right)=\proR_{x}\left(\left(\left|X_{1}+1\right|,\,Y_{1}\right)=y\right)\\
% & =\proR_{x}\left(\left(X_{1,}\,Y_{1}\right)=y-\mathbf{e}_{1}\right)+\proR_{x}\left(\left(X_{1,}\,Y_{1}\right)=\overline{y}-\mathbf{e}_{1}\right)\\
% & =\frac{q\left(x,\,y-\mathbf{e}_{1}\right)}{\sum_{u\sim y-\mathbf{e}_{1}}q(u,\,y)}+\frac{q\left(x,\,\overline{y}-\mathbf{e}_{1}\right)}{\sum_{u\sim\overline y-\mathbf{e}_{1}}q(u,\,y)}=\frac{2+2}{2d+1}=\frac{4}{2d+1}
%\end{align*}
%which equals $p_{xy}.$
As we have already mentioned, by electric network reduction \citep[Section 2.3]{LyonsPeres}, we are able to say that $G_{\mathcal{Q}}(x,\,y)=g(x,\,y)$ for all $x,\,y $ and $G_{\cH}(x,\,y)=G_{\cH'}(x,\,y)$ for all $x,y\in H$. Moreover by construction $\pi_{\cH'} \equiv \pi_{\mathcal{Q}}$ on $H$ and by checking the first step transition probabilities it is easy to notice that $W'\stackrel{d}{=}\stackrel{}{L}$. Therefore,
for all $x,\,y\in H$  it holds that
\begin{align*}
G_{\cH}(x,\,y)&=  G_{\cH'}(x,\,y) =\frac{1}{\pi_{\cH'}(y)}\ER_{x}\left[\sum_{n\ge0}\one{L_{n}=y}\right]=\frac{1}{\pi_\mathcal{Q}(y)}\ER_{x}\left[\sum_{n\ge0}\one{W_{n}'=y}\right]\\
            &=\frac{1}{\pi_\mathcal{Q}(y)}\ER_{x}\left[\sum_{n\ge0}\one{W_{n}=y}\right]
            +\frac{1}{\pi_\mathcal{Q}(y)}\ER_{x}\left[\sum_{n\ge0}\one{W_{n}=\overline{y}-\mathbf e_{(1)}}\right]\\
& = G_{\mathcal{Q}}\left(x,\,y\right)+G_{\mathcal{Q}}\left(x,\,\overline{y}-\mathbf e_{(1)}\right)
\end{align*}
where the second equality uses that $W'\stackrel{d}{=}L.$ The conclusion follows immediately after using that $G_{\mathcal{Q}}(x,\,y)=g(x,\,y)$ for all $x,\,y\in \Z^d$.
\end{proof}

\begin{rem} Let $N\in\N$ and consider the set $C_N := ([0,N]\times [-N,N]^{d-1})\cap H$.
Let
\[
G_{C_N}(x,y):=\frac{1}{\pi_\cH(y)}\ER_x\left[ \sum_{m=0}^{\tau_{C_N}}\one{S_{m}=y}\right],\quad x,\,y\in H,
\]
where $\tau_{C_N}:=\inf\{m\ge 0:\,S_m\notin C_N\}.$ In fact we are looking at the Green's function of a random walk on $H$ which is killed when leaving $C_N$.
Then by the arguments of Proposition~\ref{lem:half_green} one can guess that
\eq{}\label{eq:guess}
G_{C_N}(x,y):= g_{K_N}(x,\,y) + g_{K_N}(x,\,\overline{y}-\mathbf{e}_{(1)})
\eeq{}
where $K_N:=([-N-1,N]\times [-N,N]^{d-1})\cap \Z^d= C_N\cup (\overline{C_N}-\mathbf{e}_{(1)})$ and $g_{K_N}$ is the Green's function of the simple random walk on $\Z^d$ which is killed when leaving $K_N$. Having this guess it is straightforward to verify that this is the right choice since $G_{C_N}(\cdot,\,y)$, $y\in H$, is the unique solution to
\[
\begin{cases}
\sum_{z\sim x,\, z\in H} c_{\cH}(z,x) (G_{C_N}(z,\,y)-G_{C_N}(x,\,y) )=-\delta_x(z),&x\in C_N,\\
G_{C_N}(x,\,y)=0,&x\notin C_N
\end{cases}
\]
\cite[Proposition~6.2.2]{LawlerLimic}.
Notice finally that sending $N\to+\infty$ we get back \eqref{eq:green_halfspace} as $g_{K_N}(\cdot,\,\cdot)\to g(\cdot,\cdot)$. This approach offers a concise alternative to prove Proposition~\ref{lem:half_green}, but is of course based on the ``educated guess'' \eqref{eq:guess}.
\end{rem}
\begin{rem}Another natural case which is worth comparing with \eqref{eq:green_halfspace} is the Green's function of the process $(S_n)_{n\ge 0} = (|S^{(1)}_n|,\,S^{(2)}_n,\dots, \,S^{(d)}_n)_{n\ge 0}$, where $(S^{(1)}_n,\,S^{(2)}_n,\dots, \,S^{(d)}_n)_{n\ge 0}$ is the simple random walk on $\Z^d$. It is easy to see that $S_n$ has the same law of a random walk on $H$ with conductances $c(x,y)$ equal to \nicefrac{1}{2} if $x_1 = y_1 = 0$ and equal to one otherwise. Its Green's function equals
\[
g(x,y)+ g(x,\overline{y}),\quad x,y\in H.
\]
\end{rem}

\begin{rem} The Green's function $G_{\cH}$ is not translation invariant and the maximum of $G_{\cH}(x,x)$ is on the hyperplane $\{x\in \Z^d:\,x_1 = 0\}$. More precisely it follows from \eqref{eq:green_halfspace} that
\eq{}\label{rem:sup_G_CH}
g(0,\,0)=\inf_{x\in H}G_{\cH}(x,\,x)<\sup_{x\in H}G_{\cH}(x,\, x)=G_{\cH}(0,\,0),
\eeq{}
and that $\lim_{x_{1}\to+\infty}G_{\cH}(x,x)=g(0,0)$. Notice that even though we could have proven that $\sup_{x\in H}G_{\cH}(x,\, x)=G_{\cH}(0,\,0)$ with Rayleigh's monotonicity law, we could not employ such a technique to obtain the strict inequality \eqref{rem:sup_G_CH}.
\end{rem}
\section{Green's function for the strip}\label{sec:strip}
The same idea of folding $\Z^d$ on itself allows us to obtain a closed formula for the strip $\mathcal{S}_L:=[0,\,L-1]\times \Z^{d-1}$ for $L\in \{2,\,3,\,\ldots\}$, $d\ge 4$ and nearest-neighbour bonds. The conductances are set to be $c_{\mathcal{S}_L}\equiv 1$ for all the bonds.
\begin{prop}
With the above notation one has
\begin{equation}\label{eq:green_strip}
  G_{\mathcal{S}_L}(x,\,y)=\sum_{k=-\infty}^{+\infty}\left[g(x,\,(kL+L-1)\mathbf e_{(1)}+\overline{y})\one{k\in 2\N+1}+g(x,\,kL\mathbf e_{(1)}+y)\one{k\in 2\N}\right].
\end{equation}
\end{prop}
\begin{proof}
The idea is to apply a so-called "mountain-and-valley" fold to $\Z^d$. We are splitting each of the conductances  connecting the points in $L\Z\times \Z^{d-1}$ and $L\Z\times \Z^{d-1} - \mathbf{e}_{(1)}$, which have value one, into two conductances in series with value two, then we fold $\Z^d$ along the lines $\{x_1 = k L -\nicefrac{1}{2}\}$, $k\in \Z$, as described in Figure~\ref{fig:strip}. This operation will translate a point $A_0\in \mathcal{S}_L$ into a family of points $\left\{A_k\right\}_{k\in \Z}$, where
\[
A_k:=\begin{cases}
(kL+L-1)\mathbf e_{(1)}+\overline{A_0}&k\in 2\N+1\\
kL\mathbf e_{(1)}+A_0&k\in 2\N
\end{cases}.
\]
\begin{center}
\begin{figure}[ht!]
\makebox[\textwidth][c]{\includegraphics[height=0.35\textheight]{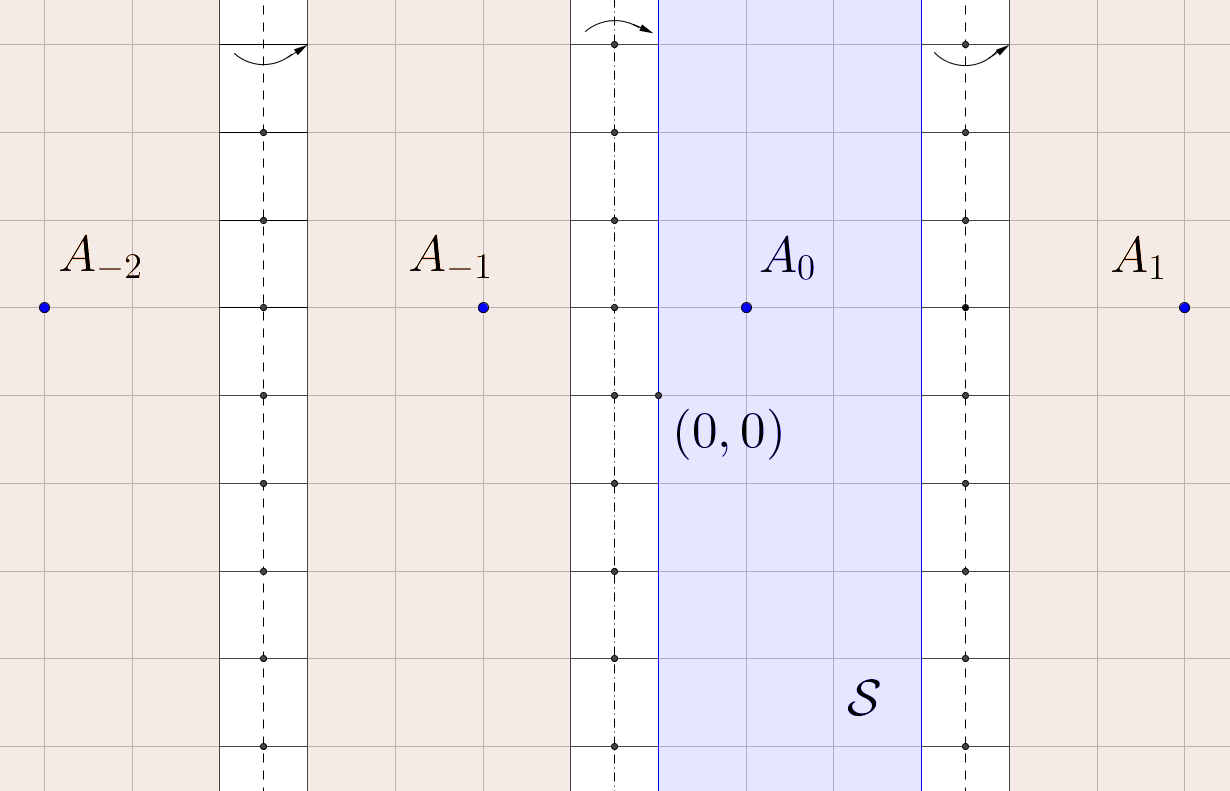}}
\caption{Following traditional origami notation, we are folding the strip and its translates in a mountain (dot dashed) and valley (dashed) fashion. The points $A_{-2},\,A_{-1},\,A_1$ are (a few of) the translates of $A_0$.}
\label{fig:strip}
\end{figure}
\end{center}
By comparing the random walk on the strip and the projection of the simple random walk onto the strip  under the above mentioned folding, one gets \eqref{eq:green_strip}.
\end{proof}
\begin{rem}[Transience on the strip]
This makes one understand that the Green's function is constant along hyperplanes of the form $\{x\in \Z^d:\,x_1=a\}$, $a=0,\,\ldots,\,L-1$. Note also that our formula combined with the estimate for the transient simple random walk (cf. \citet[Theorem~4.3.1]{LawlerLimic})
\begin{equation}
\max\left\{ 1,\,c_{\ell}\left\Vert x-y\right\Vert ^{2-d}\right\} \le g(x,\,y)\le\max\left\{ 1,\,c_{r}\left\Vert x-y\right\Vert ^{2-d}\right\} ,\quad c_{\ell,}\,c_{r}>0,\;x,\,y\in \Z^d\label{eq:lawler}
\end{equation} implies that the Green's function is finite on the diagonal in $d\ge 4$, that is, the random walk is transient on $\mathcal{S}_L$.
\end{rem}

\section{Green's function on the orthant}\label{sec:orthant_prel}

Let $\cO$ be the subgraph of the $d$-dimensional lattice with vertex set
\[
O:=\left\{ x\in\Z^{d}:\,\forall\,i=1,\dots,d:\,x_{i}\geq0\right\} =\N_0^{d}
\]
and nearest-neighbor bonds. This graph is also known with the name of discrete orthant (called ``octant'' in $d=3$). We set the bonds of $\cO$ to have $c_\cO \equiv 1$.

For $d\geq 3$, the Green's function of a random walk $(S_{n})_{n\ge0}$ on $\cO$
is given by
\[
G_{\cO}(x,\,y):=\frac{1}{\pi_{\cO}(y)}\ER_{x}\left[\sum_{n\geq0}\one{S_{n}=y}\right],\quad x,y\in\cO,
\]
where $\pi_\cO(x) := \sum_{y\sim x} c_\cO (x,y)$ as usual.

We wish to prove a closed formula for the Green's
function not only for the orthant, but also for more general subgraphs of the lattice
in which $m$ components are non-negative. We denote by $\mathcal{U}_m$ the graph with vertex set $\N_0^m\times \Z^{d-m}$ and with nearest-neighbor unitary conductances. We call their Green's function $G_m$ in place of $G_{\mathcal{U}_m}$ to ease the notation.  Also notice that $G_0(\cdot,\,\cdot) \equiv g(\cdot,\,\cdot)$ and $G_d(\cdot,\,\cdot)\equiv G_\cO(\cdot,\,\cdot)$.
\begin{prop}[Green's function on the orthant]\label{lem:octo_green} For all $x,\,y \in \mathcal{U}_m$
\begin{equation}\label{eq:green_mhyper}
G_{m}(x,\,y)=\sum_{v\in\left\{ 0,\,1\right\} ^{m}\times\left\{ 0\right\} ^{d-m}}g\left(x,\,\left((-1)^{v_{i}}\left(y_{i}+\nicefrac{1}{2}\right)-\nicefrac{1}{2}\right)_{i=1}^{d}\right).
\end{equation}
In particular, for all $x,\,y\in\mathcal{\cO},$
\eq{}\label{eq:green_mhyperO}
G_{\cO}(x,\,y)=\sum_{v\in\left\{ 0,\,1\right\} ^{d}}g\left(x,\,\left((-1)^{v_{i}}\left(y_{i}+\nicefrac{1}{2}\right)-\nicefrac{1}{2}\right)_{i=1}^{d}\right).
\eeq{}
\end{prop}
\begin{figure}
\begin{center}\includegraphics[width=0.3\textwidth]{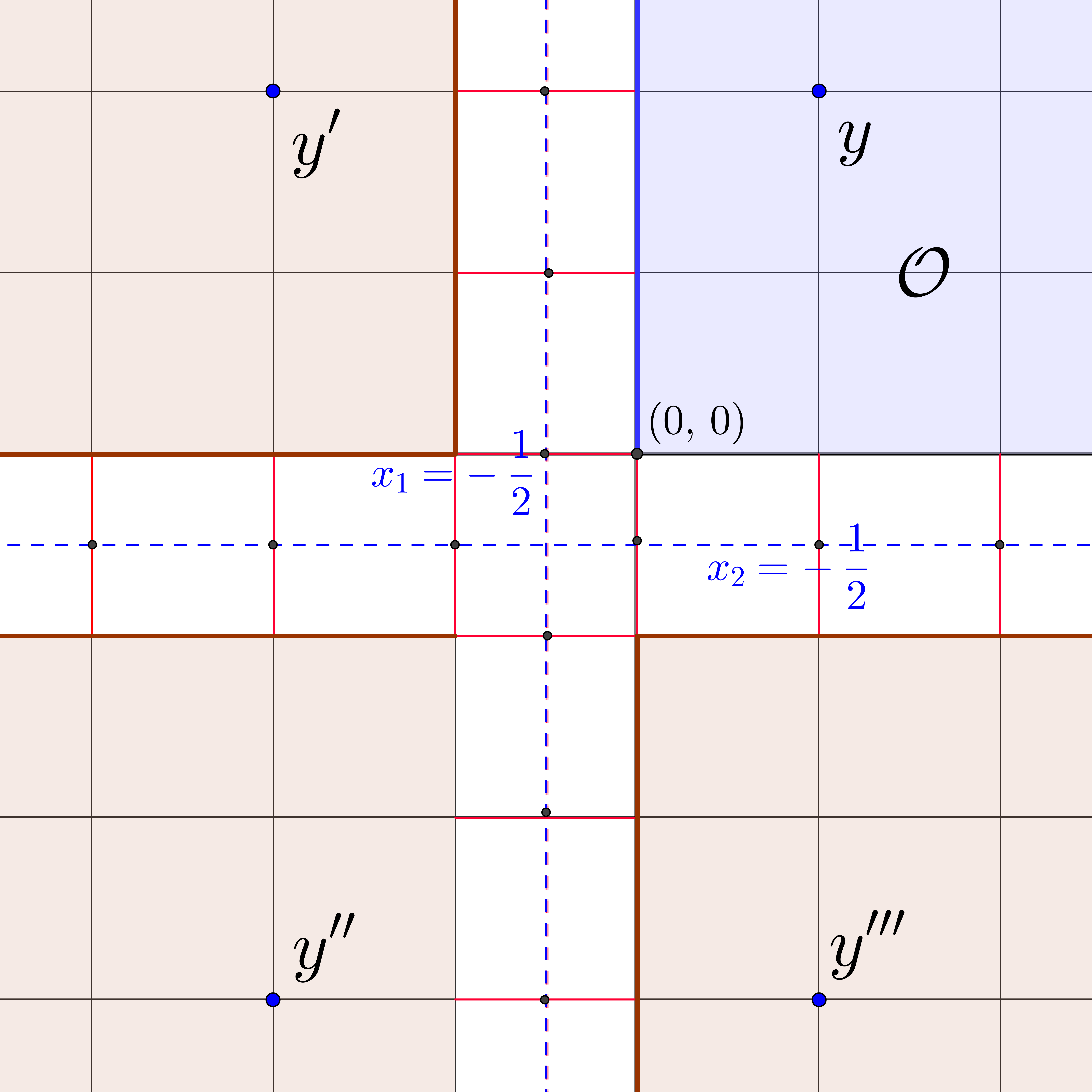}
\end{center}
\caption{For $y\in \cO$, $\left\{\left((-1)^{v_{i}}\left(y_{i}+\nicefrac{1}{2}\right)-\nicefrac{1}{2}\right)_{i=1}^{2}:\,v\in\left\{ 0,\,1\right\}^{2}\right\}=\{y,\,y',\,y'',\,y'''\}$, in this two-dimensional example.}
\label{fig:Reflections}
\end{figure}
\begin{proof} Before we begin, we want to stress that the apparently complicated formulas \eqref{eq:green_mhyper} and \eqref{eq:green_mhyperO} are nothing but a sum over all the reflections of the point $y$ about $m$ axes of the form $\left\{ x\in\R^{d}:\,x_{j}=-\nicefrac{1}{2}\right\} $
for some $1\leq j\leq d$.  Figure~\ref{fig:Reflections} clarifies this in the case $m=d$.

The proof is similar to that of Proposition \ref{lem:half_green} so we
will only sketch it here. The notation we adopt is also similar to
stress we are essentially going over the same argumentation. Since
the orthant is a special case of intersections of $d$ half spaces, we will work directly for a subspace $\mathcal{U}_m$ and $m\geq 1$, being $\mathcal{U}_0 = \Z^d$ trivial.

To $\mathcal{U}_m$, we add all the bonds of length $\nicefrac{1}{2}$ that connect the ``face'' $\mathcal{F}_{j}:=\left\{ x\in \mathcal{U}_m :\,x_{j}=0\right\}$ to the shifted ``face'' $\mathcal{F}_{j}-\nicefrac{1}{2}\mathbf{e}_{(j)}$
for all $1\leq j\leq m$ and we put on each newly added edge a conductance
equal to $2$. Call this new graph $\mathcal{U}_m'$ and its Green's function $G'_m$. Clearly $G'_m(x,y)=G_m(x,y)$ for all $x,y\in \mathcal{U}_m$. Denote by $L$ be the random walk on $\mathcal{U}_m'$
driven by such conductances.

At this point we modify the discrete lattice in a similar way as in Proposition \ref{lem:half_green}. Essentially for all $1\leq j\leq m$ we replace each conductance which connects the hyperplanes $\mathcal{I}_j:=\{x\in \Z^d\,:\, x_j = 0\}$ and $\mathcal{I}_j-\mathbf{e}_{(j)}$ by two conductances in series and value two (these are the red bonds in Figure~\ref{fig:Reflections}). These new conductances have length $\nicefrac{1}{2}$ and connect  $\mathcal{I}_j-\nicefrac{1}{2}\mathbf{e}_{(j)}$ to either $\mathcal{I}_j$ or $\mathcal{I}_j-\mathbf{e}_{(j)}$ for some $1\leq j\leq m$. Call this new graph $\mathcal{Q}$.

Let $X=(X_{n})_{n\ge0}=( X^{(1)}_{n},\,\ldots,\, X_{n}^{(d)})_{n\ge0}$
be the random walk on $\mathcal{Q}$ starting in $\mathcal{U}_m$ and $G_\mathcal{Q}(\cdot,\,\cdot)$ its Green's function. Let $Y$ be the reflection of
$X$ on the hyperplanes given by $\left\{ x\in\R^{d}:\,x_{j}=-\nicefrac{1}{2}\,\mathbf{e}_{(j)}\right\}$, $1\leq j\leq m$, that is, $Y_{0} = X_{0}$ and $Y_{n}=(Y_{n}^{(1)},\,\ldots,\, Y_{n}^{(d)})$
with
\[
Y_{n}^{(k)}:=\begin{cases}
(-X_{n}^{(k)}-1)\one{X^{(k)}_{n}\le-\nicefrac{1}{2}}+X_{n}^{(k)}\,\one{X_{n}^{(k)}>-\nicefrac{1}{2}} & 1 \leq k \leq m,  \\
 X_{n}^{(k)} & \mbox{otherwise}
\end{cases}.
\]
We can
now use the fact that $G_m \equiv G'_m$ on $\mathcal{U}_m\times \mathcal{U}_m$, that $G_\mathcal{Q}\equiv g$ on $\Z^d\times \Z^d$ and the equivalence of the laws of the random walks $L$ and
$Y$ to show that for all $x,y\in \mathcal{U}_m$
\begin{align}
  G_m(x,\,y) &=\sum_{v\in\left\{ 0,\,1\right\} ^{m}\times\left\{ 0\right\} ^{d-m}}g\left(x,\,\left((-1)^{v_{i}}\left(y_{i}+\nicefrac{1}{2}\right)-\nicefrac{1}{2}\right)_{i=1}^{d}\right).\label{eq:formula_one}
\end{align}

\end{proof}

We are interested now in monotonicity properties of Green's functions. We could not find in the literature a reference to the next Lemma, so we decided to give a short proof for it. Let $x,\,y\in \Z^d$ and define the partial relation $x\succeq y$ if and only if $|x_i|\geq|y_i|$ for all $1 \leq i \leq d$. This is also known as product order.
\begin{lem}[Monotonicity of $g(0,\,\cdot)$ with respect to the product order]
\label{lem:product_order}If $x,\,y\in\Z^d$ and $x\succeq y$, then
$g(0,\,x)\le g(0,\,y)$.\end{lem}
\begin{proof}
We have, from \citet[Eq. (2.10)]{Montroll}, that \[
2d g(0,\,x)=\int_{0}^{+\infty}\e^{-t}\prod_{i=1}^{d}I_{x_{i}}\left(\frac{t}{d}\right)\De t.\]
For $j,\,j'\in\N_{0},$ one has $I_{j}(t)\ge I_{j'}(t)$ for all $t\in[0,\,+\infty)$
if $j'\ge j$ . Considering also that $I_{-m}=I_{m}$ for $m\in\Z$
\citep[Eq. 9.6.6]{AbrSte}, the product order yields the desired conclusion.
\end{proof}
\begin{cor}
$G_m(x,\,\cdot)$ is monotone decreasing with respect to the product order for all $x\in \mathcal U_m$.
\end{cor}
\begin{proof}
The result follows combining Proposition \ref{lem:octo_green} with Lemma~\ref{lem:product_order}.
\end{proof}

\begin{rem}
From \eqref{eq:formula_one} and Lemma \ref{lem:product_order} above one obtains the location of the maximum of the Green's function:
\begin{equation}
\sup_{x\in\mathcal{U}_m}G_{m}(x,\,x)=G_m(0,\,0)\label{eq:sup_octo},\quad 0\leq m \leq d.
\end{equation}
Another consequence of \eqref{eq:formula_one} is the following chain of strict inequalities:
\[
g(x,\,y)< G_1(x,\,y) < \ldots < G_d(x,\,y), \quad x,\,y\in \cO.
\]
More precisely $ G_j(x,\,y) < G_{j+1}(x,\,y)$ for all $x,y \in \mathcal{U}_{j+1}$ and $0\leq j\leq d-1$.
\end{rem}

\subsection{A useful formula\label{subsec:useful} at the origin} An interesting consequence of our analysis is that we can explicitly calculate (\ref{eq:formula_one}) in the case $x=y=0.$ Namely we show
\begin{lem}\label{lem:formula_G_int} Let $I_k(\cdot)$ be the modified Bessel function of the first kind of order $k\in \N_0$. For all $0\le m\le d$,
\begin{equation}
G_{m}(0,\,0)=\frac{1}{2d}\int_{0}^{+\infty}\e^{-x}\left(I_{1}\left(\frac{x}{d}\right)+I_{0}\left(\frac{x}{d}\right)\right)^{m}I_{0}\left(\frac{x}{d}\right)^{d-m}\De x.\label{eq:formula_G_int}
\end{equation}
\end{lem}
\begin{proof}
Let $\gamma_{j}:=\sum_{k=1}^{j}\mathbf{e}_{(k)}$ for $1\le j\le d.$
The formula \eqref{eq:green_mhyper} is telling us that, to compute $G_m(0,0)$ we
have to choose, for each $j\in\{0,...,m\}$, $j$ hyperplanes out of $m$ about which to reflect the
point $0$, and then compute the sum of terms of the form $g(0,z)$, where $z$ is one reflection of the origin about these hyperplanes. However, the value of $g(0,z)$ is independent
of the $j$ hyperplanes chosen, due to the fact that $g(x,\,y)$ depends only on $\|x-y\|$. This yields
\begin{equation*}
G_{m}(0,\,0)=\sum_{j=0}^{m}\binom{m}{j}g\left(0,\,\gamma_{j}\right).\label{eq:formula_binomial}
\end{equation*}
As a consequence of \citet[Eq. (2.11b)]{Montroll} we obtain
\begin{equation*}
g\left(0,\,\gamma_{j}\right)=\frac{1}{2d}\int_{0}^{+\infty}\e^{-x}I_{1}\left(\frac{x}{d}\right)^{j}I_{0}\left(\frac{x}{d}\right)^{m-j}I_{0}\left(\frac{x}{d}\right)^{d-m}\De x,\quad j=0,\,\ldots,\,m
\end{equation*}
whence \eqref{eq:formula_G_int}.
\end{proof}

One can use the above formula as a starting point to show asymptotic expansions of $G_\cO$ for large values of $d$. Furthermore, it appears to be useful to get statements pointwise in the dimension. The corollary below provides a simple example.

\begin{cor} $2d G_{\cO}(0,\,0)$ is decreasing in $d$ for all $d\geq 3$.
\end{cor}
\begin{proof}
This is an immediate consequence of Lemma~\ref{lem:formula_G_int}. Indeed for $d'\ge d$, \citet[Eq. 9.6.19]{AbrSte} gives that
\begin{align*}
\left(I_{0}\left(\frac{x}{d}\right)+I_{1}\left(\frac{x}{d}\right)\right)^{d}= & \left(\int_{0}^{\pi}\e^{\frac{x}{d}\cos\theta}\frac{\left(\cos\theta+1\right)}{\pi}\De\theta\right)^{d}\\
 & \ge\left(\int_{0}^{\pi}\e^{\frac{x}{d'}\cos\theta}\frac{\left(\cos\theta+1\right)}{\pi}\De\theta\right)^{d'}=\left(I_{0}\left(\frac{x}{d'}\right)+I_{1}\left(\frac{x}{d'}\right)\right)^{d'}
\end{align*}
where the second line follows from Jensen's inequality and the fact
that the measure $\pi^{-1}\left(\cos\theta+1\right)\De\theta$ has
mass $1$. Plugging this into \eqref{eq:formula_G_int} with $m=d$, we can conclude.
\end{proof}
\bibliographystyle{abbrvnat}
\bibliography{literatur}

\end{document}